%% file: 0_Main.tex
\documentclass[abstracton, paper=a4, final, fontsize=11pt,DIV=12,bibliography=totoc]{scrartcl}

\input{header.tex}

\begin{document}

\title{Derivation of the Monokinetic Vlasov-Stokes equations}
\author[1]{Richard M. H\"ofer\thanks{richard.hoefer@ur.de}}
\author[2]{Amina Mecherbet\thanks{mecherbet@imj-prg.fr}}
\author[3]{Richard Schubert\thanks{schubert@iam.uni-bonn.de}}
\affil[1]{Faculty of Mathematics, University of Regensburg, Germany}
\affil[2]{Institut de Mathématiques de Jussieu -- Paris Rive Gauche, Université Paris Cit\'e, France}
\affil[3]{Institute for Applied Mathematics, University of Bonn, Germany}

\selectlanguage{english}

\maketitle

\begin{abstract}
   We consider a microscopic model of spherical particles with inertia in a Stokes flow. As the particle number grows to infinity and their size goes to zero we derive the monokinetic Vlasov-Stokes equations as mean-field limit. We do this under the assumption that the particles have initial velocities given by a Lipschitz velocity profile and prove the mean-field limit for times of the order of the inverse Lipschitz constant. Notably this is not a perturbative result. In particular, we do not require the inertia of the particles to vanish in the limit. Thereby the result improves upon the perturbative derivation in \cite{HoferSchubert23b} in the case of a monokinetic limit density.\\


\end{abstract}

\setcounter{tocdepth}{1}
\tableofcontents



\input{1_Introduction.tex}

\input{section1.tex}

\appendix
\input{appendix.tex}

\section*{Acknowledgements}

The authors are grateful for the opportunity of two intensive one-week stays at Mathematisches Forschungszentrum Oberwolfach as Oberwolfach Research Fellows which was the starting point of this article. The authors were supported by the project Suspensions, grants ANR-24-CE92-0028 of the French National Research Agency (ANR), HO 6767/3-1 and SCHU 3800/1-1 of the German  Research Foundation (DFG).

 \begin{refcontext}[sorting=nyt]
\printbibliography
 \end{refcontext}

\end{document}

%% file: header.tex
\usepackage[utf8]{inputenc}
\usepackage[T1]{fontenc}
\usepackage{lmodern}
\usepackage[english, french]{babel}
\usepackage{csquotes}
\usepackage{amsmath}
\usepackage{amsthm, amssymb, amsmath, amsfonts, mathrsfs, dsfont, esint,textcomp}
\usepackage{thmtools}
\usepackage[colorlinks=true, pdfstartview=FitV, linkcolor=blue, citecolor=blue, urlcolor=blue,pagebackref=false]{hyperref}
\usepackage{mathtools}
\usepackage[shortlabels]{enumitem}

\usepackage[backend=biber,giveninits,maxnames=5, maxalphanames=5,style=alphabetic,isbn=false, date=year, doi=false, eprint= true, url= true, sorting=ynt, sortcites = true]{biblatex}

 \usepackage{authblk}
  
\setkomafont{date}{\large}

\addtokomafont{disposition}{\rmfamily}

\usepackage{microtype}

\usepackage[notcite,notref,color]{showkeys}
\definecolor{labelkey}{gray}{.8}
\definecolor{refkey}{gray}{.8}

\definecolor{darkblue}{rgb}{0,0,0.7} 
\definecolor{darkred}{rgb}{0.9,0.1,0.1}
\definecolor{darkgreen}{rgb}{0,0.5,0}

\makeatletter
\def\cleartheorem#1{%
    \expandafter\let\csname#1\endcsname\relax
    \expandafter\let\csname c@#1\endcsname\relax
}
\makeatother
\cleartheorem{proof}
\declaretheoremstyle[
    spaceabove=6pt, 
    spacebelow=6pt, 
    headfont=\normalfont\bfseries, 
    bodyfont = \normalfont,
      qed=\qedsymbol,
]{myproofstyle} 
\declaretheorem[name={Proof}, style=myproofstyle, unnumbered]{proof}

\newtheorem{thm}{Theorem}[section]

\newtheorem{prop}[thm]{Proposition}
\newtheorem{lem}[thm]{Lemma}
\newtheorem{cor}[thm]{Corollary}
\theoremstyle{definition}

\newtheorem{rem}[thm]{Remark}

\newcommand\norm[1]{\left\lVert#1\right\rVert}

\newcommand\bra[1]{\left({#1}\right)}

\newcommand\abs[1]{\left\lvert#1\right\rvert}

\renewcommand{\le}{\leqslant}

\renewcommand{\leq}{\leqslant}
\renewcommand{\geq}{\geqslant}
\newcommand{\ls}{\lesssim}
\newcommand{\gs}{\gtrsim}
\renewcommand{\subset}{\subseteq}

\newcommand{\N}{\mathbb{N}}

\newcommand{\1}{\mathbf{1}}
\newcommand{\R}{\mathbb{R}}

\renewcommand{\d}{{\mathrm{d}}}

\renewcommand{\div}{\mathop{\rm div}}

\newcommand{\dd}{\, \mathrm{d}}

\newcommand{\dmin}{d_{\mathrm{min}}}

\renewcommand{\varrho}{\rho}

\newcommand{\w}{\mathrm{w}}

\def\dis{\displaystyle}

\DeclareMathOperator{\Id}{Id}

\DeclareMathOperator{\dv}{div}

\DeclareMathOperator{\dist}{dist}

\numberwithin{equation}{section}

\mathtoolsset{showonlyrefs}

\bibliography{bib.bib}

%% file: 1_Introduction.tex
\section{Introduction}

The Vlasov-Stokes equation is a variant of the Vlasov-Navier-Stokes model which is commonly used to describe aerosols and thus appears in a variety of applications like the modeling of climate or combustion engines. In the present contribution we seek to rigorously justify the Vlasov-Stokes equation, starting from a microscopic model, that resolves individual particles.    
For $N$ spherical particles with radius $R>0$ and center $X_i\in \R^3$, occupying $B_i=B(X_i,R), i=1,\dots,N$, and a (fluid) velocity field $u_N:\R^3\to \R^3$, we consider the following system 
\begin{equation}\label{eq:StokesN}
\left\{
\begin{array}{rcll}
-\Delta u_N+\nabla p_N &=& 0\quad& \text{on } \R^3 \setminus \bigcup_i \overline{B_i}, \\
\div u_N &=& 0\quad& \text{on } \R^3 \setminus \bigcup_i \overline{B_i}, \\u_N&=& V_i\quad& \text{on } \overline{B_i},\\
u_N(x)&\to& 0\quad& \text{as }|x|\to \infty,
\end{array}
\right.
\end{equation}
completed with 
\begin{equation}\label{eq:ODEsN}
   \begin{array}{rl}
   \dot{X}_i= V_i, \qquad &\dot{V}_i= g+\dis N \int_{\partial B_i}\sigma(u_N,p_N) n\eqqcolon g -NF_i,\\
    \end{array}
\end{equation}
and the initial conditions 
\begin{equation}\label{eq:ICN}
   \begin{array}{rl}
   X_i(t=0)=X_i^0, \qquad &V_i(t=0)= V_i^0. 
   \end{array}
\end{equation}
Here, $\sigma(u,p)=-p\Id+2Du$ is the stress tensor, $Du=\frac 12 (\nabla u+\nabla u^T)$ is the symmetric gradient, $n$ is the unit outer normal of $B_i$, and $g$ is the gravitational force. This system describes particles immersed in Stokes flow, with a no-slip condition for the velocity at the particle boundary. The particle velocities evolve according to the viscous forces acting on the particles. \\

Monitoring the empirical phase-space density of the particles, $f_N=\frac 1N\sum_{i=1}^N \delta_{X_i,V_i}$ and assuming $R \to 0$, $6 \pi N R \to 1$ and $f_N\rightharpoonup f$ as $N\to \infty$,  (at least for $t=0$), one expects on a mesoscopic scale, the Vlasov-Stokes equation for the phase-space density
\begin{equation}\label{eq:VS}
    \left\{
\begin{array}{rcll}
\partial_t f +\div_x (vf)+\div((g+u-v)f) &=&0,& \text{on } (0,+\infty)\times \R^3\times \R^3 \\
-\Delta u+\nabla p &=& \int(v-u)f\d v,& \text{on } (0,+\infty)\times \R^3  \\
\div u &=& 0,& \text{on } (0,+\infty)\times \R^3,\\
f(t=0)&=&f_0,& \text{on } \R^3\times \R^3
\end{array}
\right.
\end{equation}
Unfortunately, it is notoriously difficult to rigorously derive this equation or close relatives, like the Vlasov-Navier-Stokes equation, from a microscopic fluid-particle system. The Vlasov-Fokker-Planck-Navier-Stokes equation has been derived in \cite{FlandoliLeocataRicci19, FlandoliLeocataRicci21} starting from an intermediate (not fully microscopic) system, while the Vlasov-Navier-Stokes System was formally justified starting from two coupled Boltzman equations in \cite{BernardDesvillettesGolseRicci17,BernardDesvillettesGolseRicci18}. For the justification of the homogenized fluid equation in the presence of particles without coupling to the movement of the particles, the literature has a lot more depth and for an overwiev we refer to \cite{Hofer23}. Recently, in \cite{HoferSchubert23b}, two of the authors were able to perturbatively derive \eqref{eq:VS} from the microscopic system \eqref{eq:StokesN},\eqref{eq:ODEsN},\eqref{eq:ICN} by introducing the dimensionless parameter $\lambda$ related to the particles' strength of inertia, and replacing equation $\dot{V}_i=g- N F_i$ by
\begin{align*}
    \dot{V}_i=\lambda\left (g- N F_i \right ).
\end{align*}
The statement, that the microscopic system is close to \eqref{eq:VS} for the simultaneous limit $N\to\infty$ and $\lambda\to \infty$ builds on the fact that the formal limit $\lambda\to \infty$ (vanishing inertia) yields the first order system
\begin{align}\label{eq:firstorder}
	\dot{X}_i= V_i, \qquad F_i=\frac 1N g,
\end{align}
The mean-field behaviour of the system \eqref{eq:StokesN}, \eqref{eq:firstorder} (together with initial conditions for the particle positions) is given by the transport-Stokes equation \cite{JabinOtto04, Hofer18MeanField, Mecherbet19, Hofer&Schubert} for the number density $\rho(x)$ of the particles,
\begin{align}\label{eq:TS}
	 \left\{
\begin{array}{rcll}
\partial_t \rho +\div (\rho (u+v_s)) &=&0,& \text{on } (0,+\infty)\times \R^3 \\
\partial_t u + \nabla p &=&\rho g,& \text{on } (0,+\infty)\times \R^3 \\
\div u &=& 0,& \text{on } (0,+\infty)\times \R^3  \\
\rho(t=0)&=&\rho_0,& \text{on } \R^3,
\end{array}
\right.
\end{align}
where $v_s$ is a constant vector characterizing the self-interaction of the particles. For effects of non-spherical particles, we also refer to \cite{Duerinckx23}. \\

Instead of starting with the inertialess microscopic system, the simultaneous limit $N,\lambda\to \infty$ for particles with inertia \eqref{eq:StokesN},\eqref{eq:ODEsN},\eqref{eq:ICN} is considered in \cite{HoferSchubert23}, yielding again \eqref{eq:TS} in the limit. However, ignoring the particle inertia leads to a systematic error of order $\lambda^{-1}$ in the approximation of the microscopic system (with inertia) by the macroscopic system (without inertia). The article \cite{HoferSchubert23b} shows, that this error can be avoided by comparing to the right mesoscopic limit \eqref{eq:VS}, which, contrary to the transport-Stokes equation, takes into account particle inertia. The derivation in \cite{HoferSchubert23b} works for absolutely continuous phase-space densities (although treating monokinetic densities seems possible with a similar approach) and unfortunately only in a perturbative setting (i.e. for $\lambda\to \infty$ as $N\to \infty$) due to difficuties in the control of particle concentration. \\

The present article shows that the consideration of monokinetic phase-space densities, i.e. densities of the form $f(d x,\d v)=\rho(\d x)\otimes \delta_{w(x)}(\d v)$ yields enough control to pass to the limit $N\to \infty$ for finite, fixed $\lambda$. \\

On the microscopic level this amounts to modify \eqref{eq:ICN} to 
\begin{equation}
	V_i(t=0)=\w_0(X_i^0),
\end{equation}
where $\w_0\in W^{1,\infty}(\R^3,\R^3)$ is the initial velocity profile. On the mesoscopic level, in the case of monokinetic densities, the Vlasov-Stokes equations take the form of a pressureless Euler system
\begin{equation}\label{eq:MVS}
    \left\{
\begin{array}{rcll}
\partial_t \rho +\div (\rho \w) &=&0,& \text{on } (0,+\infty)\times \R^3 \\
\partial_t \w + \w\cdot \nabla \w &=&g+u-\w,& \text{on } (0,+\infty)\times \R^3 \\
-\Delta u+\nabla p &=& (\w-u)\rho,& \text{on } (0,+\infty)\times \R^3  \\
\div u &=& 0,& \text{on } (0,+\infty)\times \R^3  \\
\w(t=0)=\w_0,&& \rho(t=0)=\rho_0,& \text{on } \R^3 
\end{array}
\right.
\end{equation}
For the well-posedness of this system we refer to \cite{Braham25}. The Cauchy problem for the related monokinetic Vlasov-Navier-Stokes (compressible and incompressible) equation has been considered in \cite{ChoiKwan16, HuangTangZou23}.

\subsection{Notation}

For a given set of particle positions $(X_i)_{i=1}^N$ and corresponding velocities $(V_i)_{i=1}^N$, coming for example as a solution of \eqref{eq:StokesN},\eqref{eq:ODEsN}, we define the empirical measures
$$
f^N(t)=\frac{1}{N} \underset{i}{\sum} \delta_{(X_i(t), V_i(t))}, \quad \rho^N(t)= \frac{1}{N} \sum_i \delta_{X_i(t)} = \int f^N(t) d v.
$$
The minimal distance between the particles is
\begin{align*}
    \d_{\min}(t)=\min_{i\neq j} |X_i(t)-X_j(t)|:=\min_{i\neq j} d_{ij}(t).
\end{align*}
We also denote (cf. \cite{HoferSchubert23}):
 \begin{align}
 S_\beta(t)= \underset{i = 1, \cdots N}{\max}\, \underset{j \neq i}{\sum}\frac{1}{|X_i(t)-X_j(t)|^\beta},
 	\label{sums}
 \end{align}

 and we will typically suppress the dependence on $t$ in the following.
 
For any $p\in[1,\infty)$ we set
$$
|V|^p_p=\underset{i=1}{\overset{N}{\sum}}|V_i|^p, \qquad
|V|_\infty=\underset{i=1,\cdots N}{\max} |V_i|
$$
For a given probability measure $\rho\in \mathcal{P}(\R^3)$ and $\w:\R^3\to \R^3$ we denote the corresponding monokinetic phase space density by
\begin{align}
    f(t,\d x, \d v) = \rho(t, \d x) \otimes \delta_{\w(t,x)}(\d v),
\end{align}
understood in the sense that 
\begin{align}
    \int_{\R^3 \times \R^3} \varphi(x,v) \dd f(t,\d x, \d v) = 
    \int_{\R^3} \varphi(x,\w(t,x)) \dd \rho(t,\d x) \quad \text{for all } \varphi \in C_b(\R^3).
\end{align}
We will denote by $\rho$ both the measure and its density, leading to the notation $\rho\in \mathcal{P}(\R^3) \cap L^\infty(\R^3)$. Furthermore, we denote $L^p$ norms on the whole space by
\begin{align}
	\norm{\cdot}_{L^p(\R^3)}=\norm{\cdot}_{L^p}=\norm{\cdot}_{p}.
\end{align} 

\subsection{Main result}
We consider the following assumptions:
\begin{align}
NR=\frac{1}{6\pi},\tag{H1}\label{ass:gamma}\\
\underset{N}{\sup}\, \frac{S_2(0)}{N}<+\infty \label{hyp_S_2(0)} \tag{H2}
\end{align}
The main result then reads as follows.
\begin{thm}\label{thm}
Assume \eqref{ass:gamma} and \eqref{hyp_S_2(0)}, $\w_0 \in W^{1,\infty}(\R^3,\R^3)$, and $\rho_0\in \mathcal{P}(\R^3) \cap L^\infty(\R^3)$. Let $\bar{T}$ be the time of existence of a solution $(\rho,\w)\in L^\infty(0,\bar{T},\mathcal{P}(\R^3) \cap L^\infty(\R^3))\times W^{1,\infty}([0,\bar{T}]\times\R^3)$ to the monokinetic Vlasov Stokes equation \eqref{eq:MVS}.
There exists $N^\ast\in \N$, $C>0$ and $T^*\leq \bar{T}$ depending on \eqref{hyp_S_2(0)}, and monotonously on $\| \w_0\|_{W^{1,\infty}}$, such that for all $N \geq N^\star$ and $t < T^*$:
\begin{align}
\mathcal{W}_2(f^N(t),f(t)) &\leq \left((1+\|\w_0\|_{W^{1,\infty}(\R^3)})\mathcal{W}_2(\rho^N_0,\rho_0)+CN^{-1/2}\right)e^{Ct},\label{ineq:wasserstein2}\\ 
d_{\min}(t) &\geq \frac{d_{\min}(0)}{2}.\label{ineq:dmin}
\end{align}
\end{thm}
\begin{rem}
    \begin{itemize}
    \item A solution to \eqref{eq:MVS} is here a Langrangian solution $\rho$ to the transport equation for the density and a strong solution $\w$ to the pressureless Euler equation. See \cite{Braham25}
    \item Assumption \eqref{ass:gamma} simplifies the limiting equations. However it is completely unproblematic to instead assume $\lim_{N\to\infty} NR=\gamma_\ast\in (0,\infty)$ which yields a prefactor $\gamma_\ast$ in front of (both instances of) $(u-\w)$ in \eqref{eq:MVS}.
    \item It is not necessary to choose the initial velocities according to $\w_0$. It suffices that the initial empirical phase space densities converge to a monokinetic density together with  a Lipschitz type estimate $|V_i^0 - V_j^0| \leq C d_{ij}(0)$. For the sake of a leaner presentation, we do not include this generalization. 
    \item Note that, by considering the term involving $d_{\min}$, assumption \eqref{hyp_S_2(0)} implies the lower bound 
\begin{equation}\label{assump_d_min}
d_{\min}(0)\geq C N^{-1/2}.
\end{equation}
This is slightly better than the lower bound 
\begin{equation}\label{assump_d_min_old}
d_{\min}(0)\gg N^{-17/45},
\end{equation}
used in \cite{HoferSchubert23b} for the perturbative derivation of the Vlasov-Stokes system.
    \end{itemize}
\end{rem}

\subsection{Strategy}
As employed in other contributions before, the key to the proof is to control the Wasserstein distance $\mathcal{W}_2(f^N(t),f(t))$ and  the particle distances $d_{ij}$.
The main novelty with respect to \cite{HoferSchubert23b} is the recognition, that in the monokinetic case, it is possible to uniformly control  $d_{ij}$ for small times depending on the Lipschitz constant of the initial monokinetic velocity field only. 

The point is to integrate the velocity equation in \eqref{eq:ODEsN} and to use the approximate Stokes law $F_i=6\pi R(V_i-(u)_i)$ (\cite[Lemma 5.1]{HoferSchubert23}), where $(u)_i$ is some local average around $X_i$ of the velocity field that is comparable to the mean-field velocity and for which we have control of the Lipschitz constant as long as we control $S_2$ as in \eqref{hyp_S_2(0)}. We then estimate
 \begin{align*}
     |V_i-V_j|(t)  &\leq |V_i^0-V_j^0| + \int_0^t -|V_i-V_j | + | (u)_i-(u)_j|\dd s\\
     &\leq -|V_i^0-V_j^0| + \int_0^t -|V_i-V_j | + C d_{ij}\dd s.
 \end{align*}
A Gronwall type argument  yields
$$
|V_i(t)-V_j(t)| \leq  \left( |V_i^0-V_j^0|+ C\int_0^t d_{ij}(s) e^s\dd s\right) e^{-t}
$$
Since $|\frac{d}{dt} d_{ij}| \leq |V_i-V_j|$ and $V_i(0)=\w_0(X_i(0))$, another ODE argument eventually yields 
\begin{align*}
d_{ij}(t) \geq d_{ij}(0) \left( 2-e^{C t} + \frac{\|\nabla \w_0\|_\infty}{C+1} \left(e^{- t}-e^{C t}   \right) \right) 
\end{align*} 
where $C$ depends on the control of $S_2$ on $[0,t]$. Since $S_2$ itself is controlled by $S_2(0)$ and the control on $d_{ij}$, this provides 
 the desired control of the distances $d_{ij}$ and $S_2$ for short times.
 
We emphasize that in contrast to the arguments in \cite{HoferSchubert23, HoferSchubert23b}, this argument does not exploit good properties of the solution to the limit equation.
In fact we were unable to significantly improve the estimates on $d_{ij}$ through such properties. Indeed, to do so, the second order nature of the problem would naively require an estimate of the form
$$|V_i(t) - \w(t,X_i(t))| \ll \dmin.$$
This seems  impossible because due to discretization errors we expect at least $|V_i(t) - \w(t,X_i(t))| \gtrsim N^{-1/3} \gtrsim \dmin$.
Consequently, a major drawback of our result is that we are unfortunately not able to show that the time $T_\ast$, until which convergence to the limit holds, coincides with the life time $\bar T$ of the limit solution. 

The next sections are organized as follows,  in order to compare the microscopic to the mesoscopic quantities, we introduce in section \ref{section2} the main intermediate velocity fields and densities as well as some related estimates. These estimates are a slight adaptation of results from \cite{HoferSchubert23} and we provide their proof in the appendix \ref{appendix} for the sake of self-containedness. Section \ref{section3} is devoted to the proof of the main Theorem \ref{thm}.

%% file: section1.tex
\section{Preliminaries}\label{section2}

In the following proposition we gather results obtained in \cite[Lemma 5.1, Remark 5.2, Lemma 5.3]{HoferSchubert23}.
\begin{prop}\label{prop:app_Stokes}
Let $R>0$, and $4R \leq d\leq \dmin/2$. Let $A = B_d(0) \setminus B_{d/2}(0)$. Then, there exists a weight $\omega \colon A \to \R^{3\times3}$ with
\begin{align} \label{omega}
    \fint_A \omega \dd x = \Id, \qquad    \|\omega\|_\infty \leq C 
\end{align}
for some universal constant $C$, such that setting $\omega_i(\cdot)=\omega(\cdot -X_i)$ defined on $A_i = B_d(X_i)\setminus B_{d/2}(X_i)$, and with $u_N$ as in \eqref{eq:StokesN}, we have
\begin{align}\label{eq:br}
    F_i=6\pi R\left(V_i-\fint_{A_i}\omega_i u_N\right).
\end{align}
 \begin{align}\label{eq:estimate_mean_u^N_A_i}
 \left|\fint_{A_i}\omega_i u_N \right| &\lesssim \underset{j \neq i}{\sum} \frac{|F_j|}{d_{ij}}+ R^3 S_2 S_4^{1/2} |F|_2,\\
\label{eq:estimate_mean_u^N_A_i_Lip}
 \left|\fint_{A_i}\omega_i u_N -\fint_{A_j}\omega_j u_N \right| &\lesssim d_{ij} \left( |F|_\infty S_2+R^3 S_2 S_6^{1/2}|F|_2  \right).
 \end{align}
\end{prop}

We will use $\omega$ from this proposition in the following and set $\omega_i(\cdot)=\omega(\cdot -X_i)$. 

\begin{cor}\label{blub} Under the assumptions of Proposition \ref{prop:app_Stokes} we have
    \begin{align*}
|NF-V|_\infty \lesssim  \frac{|V|_2}{\sqrt{N}} \left(\left(\frac{S_2}{N}\right)^{1/2}+ \frac{1}{\sqrt{N}}\frac{S_2}{N}\frac{1}{N^2d_{\min}^2} \right)
\end{align*}
\end{cor}
\begin{proof}
Using \eqref{eq:br}, \eqref{eq:estimate_mean_u^N_A_i}, \eqref{ass:gamma} and the fact that $S_4 \lesssim \dmin^{-4}$ (see e.g.  \cite[Lemma 4.8]{NiethammerSchubert19}), we have
    \begin{align*}
|NF_i - V_i|& \lesssim  |F|_2 (S_2^{1/2}+ R^3S_2S_4^{1/2})\\
&\lesssim |F|_2 \left({S_2}^{1/2}+ \frac{S_2}{N}\frac{1}{N^2d_{\min}^2} \right)
\end{align*}
Moreover, from \cite[Proposition 3.2]{HoferSchubert23} we have 
 $$|F|_2 \lesssim R |V|_2 $$
which allows to conclude.
\end{proof}
\subsection{Reminder on the 2-Wasserstein distance}
We recall that the 2-Wasserstein distance minimizes the quadratic cost among all transport plans $\pi \in \Pi(f^N,f)$ i.e. any probability measure $\pi \in \mathcal{P}(\R^6 \times \R^6)$ having $f^N$ (resp. $f$) as a first (resp. second) marginal :
$$
\mathcal{W}_2(f^N,f)^2:=\underset{\pi \in \Pi(f^N,f)}{\inf}\left\{ \int\left(|x_1-x_2|^2+|v_1-v_2|^2\right)  \dd \pi(x_1 ,v_1 ,x_2,v_2) \right\}
$$
We aim then to construct a transport plan following the trajectories of the characteristic flows. Let $F^N$ be a vector field such that $F^N(t,X_i(t))=-F_i(t)$. We introduce the following notations
\begin{equation}\label{eq:flowN}
\begin{array}{lr}
\left\{
\begin{array}{rcl}
\dot{Y}^N(t,X_i^0)&=&W^N(t,X_i^0)=V_i(t) \\
Y^N(0,X_i^0)&=& X_i^0
\end{array}
\right.,&
\left\{
\begin{array}{rcl}
\dot{W}^N(t,X_i^0)&=&\dot{V}_i(t) = g+N F^N(t,Y^N(t,X_i^0))\\
W^N(0,X_i^0)&=&\w_{0}(X_i^0)
\end{array}
\right.
\end{array}
\end{equation}
such that 
\begin{equation}\label{eq:def_F^N}
X_i(t)=Y^N(t,X_i^0), \quad V_i(t)=W^N(t,X_i^0).
\end{equation}
Note that $(Y^N,W^N)$ is only defined on the support of $\rho^N$. We will never evaluate it outside of that support.
We also define the characteristics associated to the the limit equation \eqref{eq:MVS}
\begin{equation}\label{eq:flow}
\begin{array}{lr}
\left\{
\begin{array}{rcl}
\dot{Y}(t,x)&=&W(t,x) =\w(t,Y(t,x))\\
Y(0,x)&=& x
\end{array}
\right.,
&
\left\{
\begin{array}{rcl}
\dot{W}(t,x)&=&g+u(t,Y(t,x))-W(t,x)\\
W(0,x)&=& \w_0(x)
\end{array}
\right. .
\end{array}
\end{equation}
We can construct now the transport plan $\gamma_t(\d x_1,\d v_1,\d x_2,\d v_2)$, having $f^N(t)$ as a first marginal and $f$ as second marginal as follows. Let $\lambda_0(\d x_1,\d x_2) $ be an optimal transport plan for the $\mathcal{W}_2(\rho^N(0),\rho_0)$
 Wasserstein distance which we can define as $$\lambda_0= (T,\Id) \# \rho_0$$
 with $T$ an optimal transport map (see for example \cite{Santambrogio15}, $T$ exists as soon as $\rho_0$ is absolutely continuous with respect to the Lebesgue measure).  We denote
$$
\gamma_0=(P_x^1,\w_0\circ P_x^1,P_x^2,\w_0\circ P_x^2)\# \lambda_0,
$$
where $P^1_x, P^2_x$  are the projections on $\R^3 \times \R^3$ to the first and second space variables, respectively, i.e. \begin{align}
    P_x^1(x_1,x_2) = x_1 && P_x^2(x_1,x_2) = x_2.
\end{align}
We will in the following slightly abuse the notation by also denoting $P_x^1, P_x^2$ the projections that act on $(\R^3)^4$ by
\begin{align}
    P_x^1(x_1,v_1,x_2,v_2) = x_1, && P_x^2(x_1,v_1,x_2,v_2) = x_2.
\end{align}
We observe that $\gamma_0$ is a transport plan for $(f^N(0), f(0))$ and define $\gamma_t(\d x_1,\d v_1,\d x_2,\d v_2)$ as
$$
\gamma_t= (Y^N(t)\circ P^1_x,W^N(t) \circ P^1_x, Y(t) \circ P^2_x,W(t) \circ P^2_x )\# \gamma_0.
$$
 We also set 
$$
\lambda_t= (Y^N(t)\circ P^1_x, Y(t) \circ P^2_x )\# \lambda_0,
$$  
which is a transport plan for $(\rho_N(t),\rho(t))$. We set $\eta$ the transport cost associated to  the transport plan $\gamma_t$ 
 \begin{align}\label{eq:def_eta}
 \eta(t)&=\int_{(\R^3)^4} (|x_1-x_2|^2+|v_1-v_2|^2) \gamma_t(\d x_1,\d v_1,\d x_2,\d v_2) \\
 &= \int_{(\R^3)^4} (|Y^N(t,x_1)-Y(t,x_2)|^2+| W^N(t,x_1)-W(t,x_2)|^2   ) \gamma_0 (\d x_1,\d v_1,\d x_2,\d v_2) \\
&= \int_{\R^3 \times \R^3} (|Y^N(t,x_1)-Y(t,x_2)|^2+| W^N(t,x_1)-W(t,x_2)|^2   ) \lambda_0 (\d x_1, \d x_2)
 \end{align}
 \subsection{Definition of the intermediate transport plans and velocity fields} 
In what follows we use the following notation for any $B\subset \R^3$:
$$
\delta_B= \frac{1_{B}}{|B|}
$$
For $4 R< d \leq \dmin(t)/6$, we denote
\begin{align}
    B^d_i(t) &\coloneqq B_d(X_i(t)), \label{B_i^d} \\
    f_N^d(t,\cdot) &\coloneqq \frac 1 N \sum_i \delta_{B_i^d(t)} \otimes \delta_{V_i(t)},\label{rho_N^d} \\
    \rho_N^d(t,\cdot) & =  \frac 1 N \sum_i\delta_{B_i^d(t)}.
\end{align}
We then introduce the following approximation $w_N^d  \in \dot H^1(\R^3)$ for $u_N$ defined as the solution to
\begin{align}
    -\Delta w_N^d(t,\cdot) + \nabla p = \sum_i F_i(t) \delta_{B_i^d(t)}, \quad \dv w_N^d = 0,
\end{align}
The estimates below are obtained in \cite[Lemma 6.4]{HoferSchubert23b}, we slightly adapt the proof (see Section \ref{appendix}) in order to get the following result
\begin{lem} \label{lem1}
Under the assumption that $d_{\min} \geq 8 R$, we have 
\begin{align}
\| w_N^d\|_{W^{1,\infty}(\R^3)} & \lesssim |NF|_\infty \left(1+\frac{1}{N}S_2+\frac{1}{Nd^2} \right)\label{eq:estimate_w_N^d_W^{1,infty}} \\
\left| \fint_{A_i}\omega_i (u_N-w_N^d)\right| &\lesssim  d|NF|_\infty \left(\frac{1}{N}S_2+\frac{1}{Nd^2} \right)\left(1+\frac{1}{N^{3/2}d^2}\right)\label{eq:estimate_u^N-w_N^d}\\
 \sup_{x \in \R^3} \|u_N - w_N^d \|_{L^2(B_1(x))} &\lesssim \left(N^{-1/2} d^{1/2} + N^{-5/3}\right)\left(1+\frac{1}{N}S_2 \right) |NF|_\infty.   \label{eq:u.w.L^2_loc} 
\end{align}
\end{lem}
We introduce the following velocities keeping the same notations as in \cite{HoferSchubert23b}
\begin{align} \label{w_N^d}
    -\Delta \tilde w_N^d + \nabla p = \int (v - \tilde w_N^d)  f_N^d(\cdot,\dd v), \qquad \dv \tilde w_N^d = 0,
\end{align}
and $r_i \in \dot H^1(\R^3)$, $i=1,2$ are the solutions to 
\begin{align} \label{r_1}
           -\Delta r_1 +  \nabla p = \int (v_2 - r_1) \gamma (\cdot,\d v_1, \d x_2, \d v_2), \quad \dv r_1 = 0, 
\end{align}
and
\begin{align} \label{r_2}
  -\Delta  r_2 + \nabla p = \int (v_2 - \tilde w_N^d) \gamma^d(\cdot,\d v_1,\d x_2,\d v_2), \qquad \dv r_2 = 0.
\end{align}
Here, $\gamma^d=\gamma^d_t$ is constructed precisely as in \cite[p. 44]{HoferSchubert23b}:
We consider the
transport map  $T_d$ from $\rho_d$ to $\rho_N$ defined through $T_d(y) \coloneqq X_i$ in $B_i^d$ and the induced  plan 
$\gamma_t^{N,d} \coloneqq (T_d,\Id,\Id,\Id)\# f_N^d \in \Pi(f_N,f_N^d)$.
 By the Gluing Lemma (cf. \cite[Lemma 5.5]{Santambrogio15}), there exists a measure $\sigma \in \mathcal P((\R^3)^6)$ such that $(\pi_{1,2,3,4})\# \sigma = \gamma_t$ and $(\pi_{1,2,5,6})\# \sigma = \gamma^{N,d}$, where  $\pi$ denotes the projections on the indicated coordinates. We then denote $\gamma_t^d \coloneqq(\pi_{3,4,5,6})\# \sigma \in \Pi(f,f_N^d)$.
Since $\gamma_t^{N,d}$ is supported on $\{(x^N,v^N,x^d,v^d) : v^N = v^d\}$, we have
\begin{align}
    \int |v^d - v^N|^2 \dd \sigma(x^N,v^N,x,v,x^d,v^d) = \int |v^d - v^N|^2 \dd \gamma_t^{N,d}(x^N,v^N,x^d,v^d) = 0,
\end{align}
and thus $\sigma$ is supported on $\{(x^N,v^N,x,v,x^d,v^d) : v^N = v^d\}$.
In particular, it holds for all $\psi \in L^2_\rho(\R^3)$ that
\begin{align}
    \int \psi(x) v^N \dd \gamma_t(x^N,v^N,x,v) &= \int \psi(x) v^N \dd \sigma(x^N,v^N,x,v,x^d,v^d) \\
    &= \int \psi(x) v^d \dd \sigma(x^N,v^N,x,v,x^d,v^d) = \int \psi(x) v^d \dd \gamma^d_t(x,v,x^d,v^d).
\end{align}

We adapt the proof of  \cite[Theorem 2.6]{HoferSchubert23b} in order to obtain the following (see Section \ref{appendix} for the proof) result.
In the following, we use the standard notation for weighted spaces
\begin{align*}
	\|\varphi\|_{L^2_\rho} := \left(\int_{\R^3} |\varphi|^2 \rho \dd x\right)^{\frac 1 2}.
\end{align*}
\begin{lem} \label{lem2}
Assume that there exists a constant $\Lambda>1$ such that for all $N$ large enough
\begin{equation}\label{hyp_buckling_argument}
|NF|_\infty+|V|_\infty+ \frac{S_2}{N}\leq \Lambda , 
\end{equation}
then, there exists a constant $C_{\Lambda}$ depending monotonously on ${\Lambda}$ such that for $d=\frac{\Lambda^{-1/2}N^{-1/2}}{6}$
\begin{align}
    \|r_1-r_2\|_{L^2_\rho}+ \|r_2-\tilde w_N^d \|_{L^2_\rho} &\lesssim  \eta^{1/2}+C_{\Lambda} d,\label{eq:estimate_r_1-w_d^N} \\
    \|\tilde w_N^d - w_N^d\|_{L^2_{\rho_N^d}}+  \|\tilde w_N^d - w_N^d\|_{L^2_\rho} &\lesssim d C_{\Lambda}, \label{eq:estimate_diff_w_N^d}\\
    \|\nabla(u-r_1)\|_{L^2} + \|u - r_1\|_{L^2_\rho} &\leq  \eta^{1/2}.\label{eq:estimate_u-r_1}
\end{align}
 See \eqref{eq:def_eta} for the definition of $\eta$. 
\end{lem}

\section{Proof of the main Theorem}\label{section3}

 We recall that we aim to perform a Gronwall type estimate for the transport cost $\eta(t)$  associated to  the transport plan $\gamma_t$ defined as
 \begin{align*}
 \eta(t)&=\int_{(\R^3)^4} (|x_1-x_2|^2+|v_1-v_2|^2) \gamma_t(\d x_1,\d v_1,\d x_2,\d v_2) \\
 &= \int_{(\R^3)^4} (|Y^N(t,x_1)-Y(t,x_2)|^2+| W^N(t,x_1)-W(t,x_2)|^2   ) \gamma_0 (\d x_1,\d v_1,\d x_2,\d v_2) \\
&= \int_{\R^3 \times \R^3} (|Y^N(t,x_1)-Y(t,x_2)|^2+| W^N(t,x_1)-W(t,x_2)|^2   ) \lambda_0 (\d x_1, \d x_2)
 \end{align*}
\paragraph{Step 1: \emph{Setting of the buckling argument.}}
Note that there exists a constant $\Lambda>1$ large enough such that 
\begin{equation}\label{hyp_buckling_0}
\underset{N}{\sup}\left(|NF|_\infty(0)+ |V|_\infty(0)+ \frac{S_2(0)}{N}\right) < \Lambda/8
\end{equation}
Indeed, from \eqref{hyp_S_2(0)} we have that $\frac{S_2(0)}{N}$ is uniformly bounded, on the other hand, since
 $V_i(0)=\w_0(X_i)$ we get using Corollary \ref{blub} for $N$ large enough
$$
|NF|_\infty(0)+|V|_\infty(0)+ \frac{S_2(0)}{N} \leq C\|\w_0\|_\infty \left(\frac{S_2(0)}{N}+ 1 \right)
$$ 
which yields the desired result.
Let $0<T^*\leq \bar T$ be the maximal time such that for all $t < T^*$ and all $N$ large enough
\begin{equation}\label{hyp_buckling_argument_bis}
|NF|_\infty(t)+ |V|_\infty(t)+ \frac{S_2(t)}{N}< \Lambda
\end{equation} 
Note that \eqref{hyp_buckling_argument_bis} implies in particular that for all $t < T^*$ and all $N$ large enough, 
$
\frac{1}{Nd_{\min}^2(t)} \leq \Lambda
$ and  $d_{\min(t)}\geq 8 R$.\\
We aim then to provide a lower bound for $T^*$ independent of $N$.
\medskip
\paragraph{Step 2: \emph{Control of the Wasserstein distance.}} We prove that for all $t < T^*$ and all $N$ large enough
\begin{equation}\label{eq:control_wasserstein}
\eta(t) \leq (\eta(0)+C_\Lambda N^{-1})e^{C t}
\end{equation}
with $C_\Lambda$ depending on $\Lambda$.
We recall the definition of $\eta$ from \eqref{eq:def_eta}
$$
 \eta(t)= \int_{\R^3 \times \R^3} (|Y^N(t,x_1)-Y(t,x_2)|^2+| W^N(t,x_1)-W(t,x_2)|^2   ) \lambda_0 (\d x_1, \d x_2)
$$
One has
\begin{align*}
\frac{1}{2}\frac{\d}{\d t} \eta(t) &=  \int_{\R^3 \times \R^3} (Y^N(t,x_1)-Y(t,x_2))\cdot (W^N(t,x_1)-W(t,x_2))d \lambda_0  \\
&+(W^N(t,x_1)-W(t,x_2))\cdot \left(N F^N(t,Y^N(t,x_1))- u(t,Y(t,x_2))+W(t,x_2) \right) d \lambda_0\\
& \leq \eta + \int_{\R^3 \times \R^3} \left|N F^N(t,Y^N(t,x_1))- u(t,Y(t,x_2))+W(t,x_2) \right|^2 d \lambda_0.
\end{align*}
The last right-hand term is estimated as
\begin{align*}
   & \int_{\R^3 \times \R^3} \left|N F^N(t,Y^N(t,x_1))- u(t,Y(t,x_2))+W(t,x_2) \right|^2 \dd \lambda_0\\
   &\leq   2\int_{\R^3 \times \R^3} \left|N F^N(t,Y^N(t,x_1))- u(t,Y(t,x_2))+W^N(t,x_1) \right|^2 \dd \lambda_0\\
    &+2 \int |W^N(t,x_1)-W(t,x_2)|^2 \dd \lambda_0\\
\end{align*} 
Using \eqref{eq:def_F^N}, \eqref{eq:br} and \eqref{ass:gamma} and setting $d=\frac{\Lambda^{-1/2} N^{-1/2}}{6}$ we get
\begin{align} \label{split.force.difference}
\begin{aligned}
  & \int_{\R^3 \times \R^3} \left|N F^N(t,Y^N(t,x_1))- u(t,Y(t,x_2)))+W^N(t,x_1) \right|^2  \dd \lambda_0\\
   &\leq  \frac{1}{N} \sum_i \left| \fint_{A_i}\omega_i (u_N-w_N^d)\right|^2 + \frac{1}{N} \sum_i \left| \fint_{A_i}\omega_i (w_N^d - w_N^d(X_i))\right|^2  \\
    &+  \int_{\R^3 \times \R^3} \left|w_N^d(x_1) - w_N^d(x_2)\right|^2  \dd \lambda_t+ \|w_N^d - u\|_{L^2_\rho}^2.
\end{aligned}
\end{align}
Using \eqref{eq:estimate_w_N^d_W^{1,infty}} and \eqref{eq:estimate_u^N-w_N^d} together with \eqref{hyp_buckling_argument_bis} and the fact that $d=\frac{\Lambda^{-1/2} N^{-1/2}}{6}$
\begin{align}\label{eq:estimate_1}
&    \frac{1}{N} \sum_i \left| \fint_{A_i}\omega_i (u_N-w_N^d)\right|^2 + \frac{1}{N} \sum_i \left| \fint_{A_i}\omega_i (w_N^d - w_N^d(X_i))\right|^2 
    +  \int_{\R^3 \times \R^3} \left|w_N^d(x_1) - w_N^d(x_2)\right|^2  \dd \lambda_0 \notag \\ 
    & \lesssim C_\Lambda d^2+\eta 
\end{align}
We set
$$
    u - w_N^d = u - r_1 + (r_1 - r_2) + (r_2 - \tilde w_N^d) + (\tilde w_N^d - w_N^d).
$$
with $r_1$, $r_2$ and $\tilde w_N^d $ defined in \eqref{r_1}, \eqref{r_2} and \eqref{w_N^d}.
Plugging estimates \eqref{eq:estimate_1}, \eqref{eq:estimate_diff_w_N^d}, \eqref{eq:estimate_r_1-w_d^N}, \eqref{eq:estimate_u-r_1} in \eqref{split.force.difference} for $d=\frac{\Lambda^{-1/2} N^{-1/2}}{6}$ yields
\begin{align}
    \int_{\R^3 \times \R^3} \left|N F^N(t,x_1)- u(t,Y^N(t,x_2)))+W^N(t,x_2) \right|^2  \dd \lambda_0 \lesssim C_\Lambda d^2 + \eta.
\end{align}
Gathering all the estimates yield 
$$
\frac{d}{dt} \eta \lesssim C_\Lambda N^{-1}+ \eta
$$
which allows to conclude.
\medskip

\paragraph{Step 3: \emph{Control of the minimal distance.}} We claim that there exists a constant $T_0 \leq T^* $ depending monotonously on $\|\nabla \w_0\|_{\infty}$ and $\Lambda$ such that the following holds for all $t<T_0$ and all $N$ large enough
 \begin{equation}\label{eq:control_d_min}
d_{ij}(t) \geq  \frac{1}{2}d_{ij}(0)
\end{equation}
Indeed, we have using \eqref{eq:ODEsN} \eqref{eq:br} and \eqref{eq:estimate_mean_u^N_A_i_Lip} for all $t \leq T^*$
 \begin{align*}
    \frac{d}{dt}|V_i-V_j| &\leq - |V_i-V_j|+ | \fint_{A_i} \omega_i u_N - \fint_{A_j} \omega_j u_N| \\
     &\leq   - |V_i-V_j|+C d_{ij} \left(  \frac{S_{2}}{N}  |FN|_\infty+ R^3S_2S_6^{1/2}|F|_2 \right) \\
     &\leq   -  |V_i-V_j|+C d_{ij} \left(  \frac{S_{2}}{N}  |FN|_\infty+ R^3S_2 \frac{1}{d_{\min}^3 \sqrt{N}}|FN|_\infty  \right)\\
 &\leq   -  |V_i-V_j|+C \Lambda^2 d_{ij}
 \end{align*}
where we used $S_6\leq S_4 d_{\min}^{-2}$ and $S_4 \lesssim \dmin^{-4}$ (see e.g.  \cite[Lemma 4.8]{NiethammerSchubert19}) together with \eqref{hyp_buckling_argument_bis}. Performing a Gronwall estimate yields
\begin{equation}\label{ineq:Vij}
|V_i-V_j|(t) \leq |V_i^0-V_j^0|e^{- t}+C\Lambda^2\int_0^t e^{s-t} d_{ij}(s) ds
\end{equation}
Let us first deduce from \eqref{ineq:Vij} an upper bound for $d_{ij}$ which in turn would yield, using \eqref{ineq:Vij}, a lower bound for $d_{ij}$. We have
\begin{align*}
d_{ij}(t)&\leq d_{ij}(0)\left(1 + \|\nabla \w_0\|_{\infty}\int_0^t  e^{-s}ds\right)+ C \Lambda^2\int_0^t \int_0^s e^{\tau-s} d_{ij}(\tau) d\tau ds  \\
&=  d_{ij}(0)\left(1+ \|\nabla \w_0\|_{\infty} (1-e^{- t}) \right)+ C \Lambda^2 \int_0^t \int_\tau^t e^{\tau-s} d_{ij}(\tau)  ds d\tau  \\
&=  d_{ij}(0)\left(1+ \|\nabla \w_0\|_{\infty} (1-e^{- t})\right)+ C \Lambda^2 \int_0^t (1-e^{\tau-t}) d_{ij}(\tau)   d\tau  \\
&\leq  d_{ij}(0)\left(1+ \|\nabla \w_0\|_{\infty} (1-e^{- t}) \right)+ C \Lambda^2 \int_0^t d_{ij}(\tau)   d\tau 
\end{align*}
which yields 
$$
\left(e^{-C\Lambda^2t} \int_0^t d_{ij}(s)ds \right)' \leq  d_{ij}(0)\left(1+ \|\nabla \w_0\|_{\infty} (1-e^{- t})\right) e^{-C \Lambda^2t}
$$
hence
\begin{multline}\label{bound1}
 \int_0^t  d_{ij}(s)ds \leq e^{C\Lambda^2 t} d_{ij}(0)\int_0^t \left(1+ \|\nabla \w_0\|_{\infty} (1-e^{- s})\right) e^{-C \Lambda^2s}ds\\
 =d_{ij}(0) \left[ \left(1+\|\nabla \w_0\|_{\infty} \right) \left(\frac{e^{C\Lambda^2t}-1}{C\Lambda^2} \right)-\|\nabla \w_0\|_\infty \frac{e^{C\Lambda^2t}-e^{-t}}{1+C\Lambda^2}  \right]
\end{multline}
On the other hand proceeding much as above and inserting \eqref{bound1} yields
\begin{align*}
d_{ij}(t) 
&\geq  d_{ij}(0)\left(1-(1-e^{-t}) \|\nabla \w_0\|_{\infty} \right) -\int_0^t C \Lambda^2 d_{ij}(\tau) d\tau \\
&\geq d_{ij}(0) \left( 2-e^{C\Lambda^2 t} + \frac{\|\nabla \w_0\|_\infty}{C\Lambda^2+1} \left(e^{- t}-e^{C \Lambda^2 t}   \right) \right) ,
\end{align*}
where the term between brackets is decreasing in time so that we can find $T_0$ depending monotonously on $\|\nabla \w_0\|_\infty$ and satisfying the desired inequality \eqref{eq:control_d_min}.


\paragraph{Step 4: Control of the $|\cdot|_\infty$ norm of the drag forces and velocities.} We claim 
that for all $t< T^*$ and $N$ large enough
\begin{equation}\label{eq:control_F_V}
\begin{array}{rcl}
|NF- V|_\infty(t) &\lesssim& \displaystyle  \left( \sqrt{\frac{S_2}{N}}+\frac{1}{\sqrt{N}}\frac{S_2}{N}\frac{1}{N^2d_{\min}^2}\right) e^{Ct}\\
|V|_\infty(t) &\leq & \displaystyle |V(0)|_\infty+C\int_0^t   \left( \sqrt{\frac{S_2}{N}}+\frac{1}{\sqrt{N}}\frac{S_2}{N}\frac{1}{N^2d_{\min}^2}\right) e^{Cs} ds
\end{array}
\end{equation}
for some universal constant $C>0$. Indeed, using \eqref{eq:ODEsN} we have
\begin{align*}
\frac{1}{2} \frac{d}{dt} \underset{i}{\sum} |V_i|^2 &=\underset{i}{\sum} V_i \cdot g -\underset{i}{\sum} V_i \cdot N F_i\\
& =\underset{i}{\sum} V_i \cdot g+ N\|\nabla u_N \|_{L^2(\R^3 \setminus \bigcup B_i)}^2\\
&\leq C(|V|_2^2+\sqrt{N}|V|_2)
\end{align*}
where we have used (see  \cite[Proposition 3.2]{HoferSchubert23} or \cite[Lemma 10]{Hillairet2018} for instance)
$$
\|\nabla u_N \|_{L^2(\R^3\setminus \bigcup B_i)}^2 \leq \frac{C}{N} |V|_2^2
$$
This proves that
$$
|V|_2(t) \leq C(\sqrt{N}+|V|_2(0)) e^{Ct} \leq C\sqrt{N}(1+ \|\w_0\|_\infty) e^{Ct}
$$
Hence, combined with Corollary \ref{blub} we get for $t\leq T^*$
$$
|NF- V|_\infty(t) \lesssim \left( \sqrt{\frac{S_2}{N}}+\frac{1}{\sqrt{N}}\frac{S_2}{N}\frac{1}{N^2d_{\min}^2}\right) e^{Ct}
$$
Now using again \eqref{eq:ODEsN} we have
$$
\frac{d}{dt}|V_i|\leq - |V_i| + |NF_i-V_i| 
$$
Using Gronwall yields 
$$
|V|_\infty(t)\leq |V(0)|_\infty e^{-t}+  \int_0^t \left( \sqrt{\frac{S_2}{N}}+\frac{1}{\sqrt{N}}\frac{S_2}{N}\frac{1}{N^2d_{\min}^2}\right)e^{Cs} ds
$$
\paragraph{Step 5:\emph{ Conclusion of the buckling argument}.}
From \eqref{eq:control_d_min} and \eqref{hyp_buckling_0}, we have for all $t< T_0$ and $N$ large enough
$$
S_2(t) \leq 4S_2(0)\leq \Lambda/2
$$
Hence, using \eqref{eq:control_F_V} and \eqref{hyp_buckling_0},  there exists a constant $C>0$ such that for all $t <T_0$ and $N$ large enough
$$
|NF|_\infty(t) +|V|_\infty(t) \leq C(\sqrt{\Lambda}e^{Ct} + \|\w_0\|_\infty)
$$
Hence, one can find $T_1>0$ depending on $\|\w_0\|_\infty$ and $\Lambda$ such that for all $t\leq \min(T_0,T_1)$
$$
|NF|_\infty(t) +|V|_\infty(t) \leq \Lambda/2
$$
which proves that we can choose at least $T^*= \min(T_0,T_1)$ depending only on $\|\w_0\|_{W^{1,\infty}(\R^3)}$, and $\Lambda$ i.e. \eqref{hyp_S_2(0)}.
Eventually, \eqref{ineq:wasserstein2} follows from \eqref{eq:control_wasserstein}.

%% file: appendix.tex
\section{Proof of Lemmas \ref{lem1} and \ref{lem2}}\label{appendix}
In the following section we use the notation $\Phi$ for the Oseen tensor i.e. the fundamental solution of the Stokes equation on $\R^3$ given explicitly by
$$
\Phi(x)= \frac{1}{8\pi |x|}\left(\mathbb{I}+ \frac{x \otimes x}{|x|^2}\right).
$$
\begin{proof}[Proof of Lemma \ref{lem1}] \textbf{Step 1:} \emph{Proof of \eqref{eq:estimate_w_N^d_W^{1,infty}}.}
We split $w_N^d = \sum_i w_{N,i}$, where $w_{N,i} \in \dot H^1(\R^3)$ is the solution to
\begin{align}
         -\Delta w^d_{N,i} + \nabla p = F_i \delta_{B_i^d}, \quad \dv w^d_{N,i} = 0.
\end{align}
Then for $k = 0,1$ and all $x \in \R^3$ (e.g. by a scaling argument or explicit convolution with the fundamental solution)
\begin{align}
    |\nabla^k w^d_{N,i}(x)| \lesssim \frac{\abs{F_i}}{|x - X_i|^{1+k} + d^{1+k}}.
\end{align}

This implies         \begin{align}
        | w_N^d(x)| & \lesssim \sum_i \frac{\abs{F_i}}{|x - X_i| + d} \le \left( \frac 1N S_{1} +\frac{1}{Nd}\right) |N F|_\infty.\\
        |\nabla w_N^d(x)| & \ls \sum_i \frac{| F_i|}{|x - X_i|^2 + d^2} \le  \left( \frac 1 N S_{2}+\frac{1}{Nd^2}\right) |N F|_\infty . \quad \label{eq:nabla_w}
    \end{align}

\noindent \textbf{Step 2:} \emph{Proof of \eqref{eq:estimate_u^N-w_N^d}.}
We follow the proof of item (i) in \cite[Lemma 5.3]{HoferSchubert23}.
For given $W\in \R^3$ consider the solution $\varphi$ to
\begin{align}
    -\Delta \varphi+\nabla p= \omega_i^T W \frac{1}{\abs{A_i^d}}\1_{A_i^d}, ~ \dv \varphi=0 \quad \text{in} ~\R^3.
\end{align}
We compute
\begin{align}
\begin{aligned}
    W\cdot\fint_{A_i^d}\omega_i (w_N^d-u_N)\dd x&=\fint_{A_i^d}(w_N^d-u_N)\cdot \omega_i^T W \dd x
    = 2 \int_{\R^3} \nabla (w_N^d-u_N) \cdot D\varphi\dd x \\
    &=\sum_j \fint_{B^d_j} \varphi\dd x \cdot F_j+\sum_j \int_{\partial B_j} \varphi\cdot (\sigma(u_N,p_N)n)\dd \mathcal{H}^2\\
    &=\sum_j \int_{\partial B_j} (\varphi-\varphi_j)\cdot (\sigma(u_N,p_N)[u_N]n)\dd \mathcal{H}^2-\sum_j (\varphi_j-\varphi_j^d)\cdot F_j ,
    \end{aligned}
    \label{eq:comp_test1}
\end{align}
where $\varphi_j\coloneqq\fint_{\partial B_j} \varphi\dd \mathcal{H}^2$ and $\varphi_j^d\coloneqq\fint_{B^d_j} \varphi\dd x$.
Using that $\varphi= \Phi\ast \omega_i^T W \frac{1}{\abs{A_i^d}}\1_{A_i^d}$, the fact that $\dist(A_i^d,A_j^d)\gs d_{ij}$ for $i\neq j$, the second property from \eqref{omega}, and the decay of $\nabla \Phi$, we have for all $j\neq i$ that $|\varphi_j^d - \varphi_j| \lesssim \abs{W}d d_{ij}^{-2}$. Moreover $\|\varphi\|_\infty\ls |W|d^{-1}$ implies the bound $|\varphi_i-\varphi_i^d|\ls |W|dd^{-2}$.
Thus we get
     \begin{align}\label{eq:phi_diff}
     \begin{aligned}
       \sum_j |\varphi_j^d - \varphi_j||F_j| &\lesssim d\abs{W}  \left(\frac 1 N S_{2}+\frac{1}{Nd^{2}}\right) |N F|_\infty
       \end{aligned}
    \end{align}
    In order to estimate the first term on the right-hand side of \eqref{eq:comp_test1}, we consider the solution  $w_0$ to 
    \begin{align}\label{eq:tilde_u}
        -\Delta w_0 + \nabla p = \sum_{ j} \delta_{\partial B_{ j}} F_{ j}, ~ \dv w_0 = 0 \quad \text{in} ~ \R^3.
    \end{align}
    Here, $\delta_{\partial B_{ j}}=\frac{1}{|\partial B_{ j}|}\mathcal{H}^2_{|\partial B_{ j}}$ is the normalized Hausdorff measure. Then, by \cite[Lemma 5.4]{HoferSchubert23} and \eqref{ass:gamma}, 
    \begin{align} \label{est.u-utilde}
       \|w_0 - u_N\|^2_{\dot H^1(\R^3)} \lesssim  \| \nabla w_0\|^2_{L^2(\cup_{ j} B_{ j})} \lesssim R^2  |F|_\infty^2 S_2^2.
    \end{align}

By a classical extension result (see e.g. {\cite[Lemma 3.5]{Hofer18MeanField}}), there exists a divergence free function $\psi \in \dot H^1(\R^3)$ such that for all $1 \leq j \leq N$, $\psi = \varphi - \varphi_j$ in $B_j$ and
\begin{align}
		\|\nabla \psi\|^2_{2} \lesssim  \|\nabla \varphi \|^2_{L^2(\cup_i B_i)} \lesssim |W|^2 R^3\left(S_4 + \frac 1 {d^4} \right),
\end{align}
where we used that $\nabla \varphi$ decays like $ \nabla \Phi$.
{Using the weak formulation for $w_0$ defined as in \eqref{eq:tilde_u}, we have
\begin{align}
    \int_{\R^3} \nabla \psi \cdot D   w_0 \dd x = 0.
\end{align}
}
Hence,
\begin{align} \label{eq:double_diff1}
\begin{aligned}
    \left|\sum_j \int_{\partial B_j} (\varphi-\varphi_j)\cdot (\sigma(u_N,p_N)n)\dd \mathcal{H}^2\right|  & = \left|2 \int_{\R^3} \nabla \psi \cdot D  (u_N - w_0) \dd x\right|  \ls \norm{\nabla(u_{ N}-w_0)}_2\norm{\nabla \psi}_2 \\
    & \lesssim R^{5/2}  |F|_\infty  S_2 \left(S_4 + \frac 1 {d^4} \right)^{\frac 1 2}   |W|\\
    &\ls d \frac{ R^{3/2}}{d^2}  |F|_\infty  S_2    |W|\\
     & \ls  d |NF|_\infty  |W|\left(\frac{1}{N}S_2 \right)\frac{1}{N^{3/2}d^2} ,
     \end{aligned}
\end{align}
where we used $R<d<\dmin$, $S_4 \lesssim \dmin^{-4}$ (see e.g.  \cite[Lemma 4.8]{NiethammerSchubert19})  in the second-to-last estimate and \eqref{ass:gamma} in the last estimate. Combining \eqref{eq:phi_diff} and \eqref{eq:double_diff1} in \eqref{eq:comp_test1} and taking the sup over all $|W|=1$ finishes the proof of \eqref{eq:estimate_u^N-w_N^d}.
\\

\noindent \textbf{Step 3:} \emph{Proof of \eqref{eq:u.w.L^2_loc}.}
The proof of \eqref{eq:u.w.L^2_loc} is similar: Fix $x_0 \in \R^3$ and let $h \in L^2(B_1(x_0))$ with $\|h\|_2 \leq 1$. After extending $h$ by zero to a function on $\R^3$, consider the solution $\varphi \in \dot H^1$ to
\begin{align}
    -\Delta \varphi+\nabla p= h, ~ \dv \varphi=0 \quad \text{in} ~\R^3.
\end{align}
By regularity we have
\begin{align}\label{eq:nabla2phi}
    \|\nabla^2 \varphi\|_{2} \lesssim \|h\|_2\le 1.
\end{align}
Analogously to \eqref{eq:comp_test1}, we obtain
\begin{align}
    \int_{B_1(x_0)} h \cdot (w_N^d-u_N)\dd x 
    = \sum_j \int_{\partial B_j} (\varphi-\varphi_j)\cdot (\sigma(u_N,p_N)n)\dd \mathcal{H}^2-\sum_j (\varphi_j-\varphi_j^d)\cdot F_j.
    \label{eq:comp_test2}
\end{align}
Denoting $(\nabla \varphi)_j^d$ the corresponding average of the gradient, we observe that by symmetry $\int_{\partial B_j} (\nabla \varphi)_j^d (x - X_j) \dd \mathcal H^2(x) =0$.
We resort to the Poincar\'e-Sobolev inequality
\begin{align}
    \|\psi\|_{L^\infty (B_j^d)}  \lesssim d^{1/2} \|\nabla^2  \psi\|_{L^2(B_j^d)} \qquad \text{for all } \psi \in W^{2,2}(B_j^d) \text{ with } \int_{B_j^d} \psi = 0, ~ \int_{B_j^d} \nabla \psi = 0.
\end{align}
Hence, 
\begin{align}
    |\varphi_j-\varphi_j^d| = \left|\fint_{\partial B_j} \varphi - \varphi_j^d  - (\nabla \varphi)_j^d (x - X_j) \dd \mathcal H^2(x) \right|  \lesssim d^{1/2} \|\nabla^2  \varphi\|_{L^2(B_j^d)},
\end{align}
and using \eqref{eq:nabla2phi}, we estimate 
     \begin{align} \label{phi_diff.2}
       \sum_j |\varphi_j^d - \varphi_j||F_j| \lesssim N^{1/2} d^{1/2} \|\nabla^2  \varphi\|_{2}  | F|_\infty \ls N^{-1/2} d^{1/2}  |N F|_\infty.
    \end{align}
On the other hand, using Sobolev-embedding and again \eqref{eq:nabla2phi} yields
\begin{align}
    \|\nabla \varphi\|_{L^2(\cup_j B_j)} \lesssim \|\nabla \varphi\|_{6} (N R^3)^{1/3} \lesssim N^{1/3} R.
\end{align}
Hence, adapting the argument for \eqref{eq:double_diff1} yields
\begin{align} \label{eq:double_diff2}
\begin{aligned}
     \left|\sum_j \int_{\partial B_j} (\varphi-\varphi_j)\cdot (\sigma(u,p)n)\dd \mathcal{H}^2 \right| \ls N^{1/3} R^{2} S_2 \abs{F}_\infty 
     \ls  N^{-5/3} |NF|_\infty \left(\frac{1}{N} S_2 \right).
     \end{aligned}
\end{align}
Inserting  \eqref{phi_diff.2} and \eqref{eq:double_diff2}  into \eqref{eq:comp_test2} yields \eqref{eq:u.w.L^2_loc}.
\end{proof}

\begin{proof}[Proof of Lemma \ref{lem2}]
Note that the assumption \eqref{hyp_buckling_argument} ensures that $ 8 R \leq d\leq  d_{\min}/6$ for $d=\frac{\Lambda^{-1/2} N^{-1/2}}{6}$.\\
\noindent \textbf{Step 1:} \emph{Proof of \eqref{eq:estimate_diff_w_N^d} .}
By slightly adapting the proof of Lemma \cite[Lemma 6.5]{HoferSchubert23b} together with Sobolev embedding, we have
\begin{equation}
    \|\tilde w_N^d - w_N^d\|_{L^2_{\rho_N^d}}+ \|\tilde w_N^d - w_N^d\|_{L^2_\rho} \lesssim d\left( \frac{1}{N} S_{2}+\frac{1}{Nd^{2}}\right) |NF|_\infty.
\end{equation}
which yields the desired result using \eqref{hyp_buckling_argument} with $d=\frac{\Lambda^{-1/2}N^{-1/2}}{6}$.\\
\noindent \textbf{Step 2:} \emph{Proof of \eqref{eq:estimate_r_1-w_d^N} .}
By \cite[Eq. (6.46)]{HoferSchubert23b},  we have
$$
\|r_1-r_2\|_{L^2_\rho} \lesssim \|r_2-\tilde w_N^d\|_{L^2_\rho}
$$
It remains then to estimate the second term in left hand side of \eqref{eq:estimate_r_1-w_d^N}.

\noindent\textbf{Substep 2.1:} \emph{Splitting.} We observe that
\begin{align}
    (r_2 - \tilde w_N^d)(x) &= \int \Phi(x-z^d)(\tilde w_N^d(z^d) - v^d) - \Phi(x-z)(\tilde w_N^d(z) - v^d) \dd \gamma^d_t(z,v,z^d,v^d)  \\
    &=: \psi_1(x) + \psi_2(x) + \psi_3(x), 
\end{align}
where
\begin{align}
    \psi_1(x) &\coloneqq  \int (\Phi(x-z^d)(w_N^d(z^d) - v^d)  - \Phi(x-z)(w_N^d(z) - v^d))  \1_{\{|x - z^d| > 2d\}} \dd \gamma_t^d(z,v,z^d,v^d), \\
     \psi_2(x) &\coloneqq  \int (\Phi(x-z^d)(w_N^d(z^d) - v^d)  - \Phi(x-z)(w_N^d(z) - v^d))  \1_{\{|x - z^d| \leq 2d\}} \dd \gamma_t^d(z,v,z^d,v^d), \\
    \psi_3(x) &\coloneqq  \int \Phi(x-z^d)(\tilde w_N^d(z^d) - w_N^d(z^d)) - \Phi(x-z)(\tilde w_N^d(z) - w_N^d(z))  \dd \gamma_t^d(z,v,z^d,v^d).
\end{align}

\medskip

\noindent\textbf{Substep 2.2:} \emph{Estimate of $\psi_1$.}
We have
\begin{align}\label{eq:psi1}
\begin{aligned}
&|\psi_1(x)|\\
& \lesssim \| w_N^d\|_{1,\infty} \int |z^d-z| \left(\frac{1}{|x-z^d|}+\frac{1}{|x-z|}+\frac{1}{|x-z^d|^2}+\frac{1}{|x-z|^2}  \right)\\
&\times (1+|v^d|)  \1_{\{|x - z^d| > 2d\}} \dd \gamma_t^d(z,v,z^d,v^d)\\
&\lesssim (1+|V|_\infty) \| w_N^d\|_{1,\infty} \left( \int |z^d-z|^2 \left(\frac{1}{|x-z^d|^2}+\frac{1}{|x-z|^2}  \right)\dd \gamma_t^d(z,v,z^d,v^d) \right)^{1/2}\\
& \times \left[  \int  \1_{\{|x - z^d| > 2d\}} \dd \gamma_t^d(z,v,z^d,v^d)  + \int \left(\frac{1}{|x-z^d|^2}+\frac{1}{|x-z|^2}  \right) \1_{\{|x - z^d| > 2d\}} \dd \gamma_t^d(z,v,z^d,v^d)   \right]^{1/2}
\end{aligned}
\end{align}
%
For the second factor on the right-hand side of \eqref{eq:psi1}, due to  $d \leq \dmin/6$, we observe that 
\begin{align} 
    \sup_{x \in \R^3} \int \frac{1}{|x-z^d|^{2}} \1_{\{|x - z^d| > 2d\}} \dd \gamma_t^d(z,v,z^d,v^d) &\lesssim \frac 1 N \left(\sup_{i} \sum_{j \neq i} |X_i - X_j|^{-2}+d^{-2}\right) \\
    & \leq  \frac 1 N (S_{2}+d^{-2}) 
\end{align}
similarly since $\rho(t) \in \mathcal{P}(\R^3) \cap L^\infty(\R^3)$ 
\begin{equation}\label{eq:bound_int_sing}
  \sup_{x \in \R^3} \int \frac{1}{|x-z|^{2}}  \rho(t,\dd z) \leq C (\|\rho\|_1+\|\rho\|_\infty)
\end{equation}
Using \eqref{eq:estimate_w_N^d_W^{1,infty}}, \eqref{hyp_buckling_argument} together with Fubini and again \eqref{eq:bound_int_sing} yields
\begin{align} \label{psi.1}
\begin{aligned}
   \|\psi_1\|^2_{L^2_\rho} &\lesssim  \int |z^d-z|^2  \dd  \gamma_t^d(z,v,z^d,v^d)\\ 
    &\lesssim \int |z^N-z|^2  \dd  \gamma_t(z^N,v^N,z,v) + \int |z^N-z^d|^2  \dd  \gamma_t^{N,d}(z^N,v^N,z^d,v^d) \\
   & \lesssim \eta + \int_{\R^3} |z^d - T^d(z^d)|^2  \dd  \gamma_t^{N,d}(z^N,v^N,z^d,v^d) \\
   & \lesssim  \eta + d^2
   \end{aligned}
\end{align}
where we used the definitions of $\gamma^{N,d}$ and $T^d$ after \eqref{r_2} in the step to the last line.

\medskip

\noindent\textbf{Substep 2.3:} \emph{Estimate of $\psi_2$.}
Regarding $ \psi_2$, on the one hand we estimate, using $\|w_N^d\|_\infty \lesssim 1$ due to \eqref{eq:estimate_w_N^d_W^{1,infty}} and $d \leq d_{\min}/6$ as well as using that the balls $B_{3d}(X_i)$ are disjoint 
\begin{align}
    & \int \left(\int |\Phi(x-z^d)| (|w_N^d(z^d)| + |v^d|) \1_{\{|x - z^d| \leq 2d\}} \dd \gamma^d(z,v,z^d,v^d)\right)^2 \dd \rho(x) \\
     &  \qquad \qquad \lesssim \int \left( \frac 1 N \sum_i \fint_{B_d(X_i)} \frac{1 + |V_i|}{|x-y|}  \dd y  \1_{B_{3d}(X_i)}(x) \right)^2 \dd \rho(x) \\
     &  \qquad \qquad  \lesssim \int \frac 1 {N^2}  \sum_i \frac{1+|V_i|^2}{d^2}  \1_{B_{3d}(X_i)}(x) \dd \rho(x) 
      \\
      & \qquad \qquad \lesssim  \frac 1 {N^2d^2}  (1 + |V|^2_\infty)  \|\rho\|_1
\end{align}
For the second and third estimate we used that for fixed $x \in \R^3$ at most one term in the sum is nonzero which allows us to interchange the square and the sum.
On the other hand,
\begin{align}
   & \int \left(\int |\Phi(x-z)| (|w_N^d(z)|+  |v^d|) \1_{\{|x - z^d| \leq 2d\}} \dd \gamma^d(z,v,z^d,v^d)\right)^2 \dd \rho(x) \\
   &\qquad \qquad \lesssim \int \bra{\int |x-z|^{-2} \dd \rho(z)} (\|w_N^d\|^2_\infty+|V|_\infty^2)  \left(\int \1_{\{|x-z^d|<2d\}} \dd \rho_N^d(z^d) \right) \dd \rho(x)\\
   &\qquad \qquad\ls  (\|w_N^d\|^2_\infty+|V|_\infty^2) N^{-1}
\end{align}
where we used that for each $x \in \R^3$
$$
\int \1_{\{|x-z^d|<2d\}} \dd \rho_N^d(z^d)=\frac{1}{N}\underset{i}{\sum} \fint_{B_d(X_i)} \1_{\{|x-z|<2d\}} \dd z \leq \frac{1}{N}\underset{i}{\sum}  \1_{\{|x-X_i|<3d\}}  \leq \frac{1}{N}
$$
since the balls $B_{3d}(X_i)$ are disjoint.
Thus,  using \eqref{hyp_buckling_argument}
\begin{align} \label{psi.2}
    \|\psi_2\|^2_{L^2_\rho} \lesssim N^{-1} \lesssim d^2.
\end{align}

\medskip

\noindent\textbf{Substep 2.4:} \emph{Estimate of $\psi_3$.}
We have
    \begin{align}
        |\psi_3(x) | &=  |\Phi \ast \left((\rho_N^d - \varrho)(w_N^d -\tilde w_N^d)\right)|(x) \\
       & \left(\norm{\frac 1 {|x - \cdot|}}_{L^2_{\rho}} + \norm{\frac 1 {|x - \cdot|}}_{L^2_{\rho_N^d}}\right) \left(
        \|w_N^d -\tilde w_N^d\|_{L^2_{\rho_N^d }} + \|w_N^d -\tilde w_N^d\|_{L^2_{\varrho}}\right)
    \end{align}
Thus, by Fubini,
$$
\int \left(\norm{\frac 1 {|x - \cdot|}}_{L^2_{\rho}} + \norm{\frac 1 {|x - \cdot|}}_{L^2_{\rho_N^d}}\right)^2  \rho(x) \dd x \lesssim \underset{y}{\sup} \norm{\frac 1 {|\cdot - y|}}_{L^2_{\rho}}^2 (\|\rho_N^d\|_1+ \|\rho\|_1)
$$          
hence, thanks to \eqref{eq:bound_int_sing},  and  \eqref{eq:estimate_diff_w_N^d} we get
    \begin{align} \label{psi.3}
        \|\psi_3\|_{L^2_{\rho}} \lesssim
        \sup_y \left(\norm{\frac 1 {|\cdot - y|}}_{L^2_{\rho}}\right) \left(
        \|w_N^d -\tilde w_N^d\|_{L^2_{\rho_N^d }} + \|w_N^d -\tilde w_N^d\|_{L^2_{\varrho}}\right) \lesssim d. \qquad
        \end{align}
       
\medskip

\noindent \textbf{Step 3:} \emph{Proof of \eqref{eq:estimate_u-r_1}.}
Testing the equation for $u$ and $r_1$ by $u - r_1$ yields
\begin{equation}
    \|\nabla(u-r_1)\|_{L^2}^2 + \|u - r_1\|_{L^2_\rho}^2 =   \int (v_1 - v_2)(u - r_1) \dd \gamma (\d x_1,\d v_1, \d x_2, \d v_2) \leq \|u - r_1\|_{L^2_\rho} \eta^{1/2}
\end{equation}

\end{proof}